\documentclass{amsart}
\usepackage{latexsym,amsxtra,amscd,ifthen}
\usepackage{amsfonts}
\usepackage{verbatim}
\usepackage{amsmath}
\usepackage{amsthm}
\usepackage{amssymb}
\usepackage{url}

\theoremstyle{plain}

\newtheorem{theorem}{Theorem}
\newtheorem{lemma}[theorem]{Lemma}

\newtheorem{corollary}[theorem]{Corollary}

\numberwithin{theorem}{section}
\numberwithin{equation}{theorem}
\theoremstyle{definition}

\newtheorem{example}[theorem]{Example}

\newtheorem*{question*}{Question}

\DeclareMathOperator{\Aut}{Aut}
\DeclareMathOperator{\gr}{gr}

\begin{document}

\title{An isomorphism lemma for graded rings}

\author{Jason Bell and James J. Zhang}

\address{Bell: Department of  Pure Mathematics,
University of Waterloo, Waterloo, ON N2L 3G1,
Canada}

\email{jpbell@uwaterloo.ca}

\address{Zhang: Department of Mathematics, Box 354350,
University of Washington, Seattle, Washington 98195, USA}
\email{zhang@math.washington.edu}

\begin{abstract}
Let $A$ and $B$ be two connected graded algebras finitely
generated in degree one. If $A$ is isomorphic to $B$ as 
ungraded algebras, then they are also isomorphic to each 
other as graded algebras.
\end{abstract}





\maketitle

\section*{Introduction}
\label{xxsec0}

The isomorphism problem has been studied by several researchers in recent years
\cite{BJ, CPWZ, Ga}. In this paper we prove a result which is useful for some 
of the isomorphism problems concerning graded algebras.

Throughout this paper, we let $k$ be a base field, and all vector spaces, algebras, and morphisms are 
over $k$. 

\begin{theorem}
\label{xxthm1}
Let $A$ and $B$ be two connected graded algebras finitely
generated in degree 1. If $A\cong B$ as 
ungraded algebras, then $A\cong B$ as graded algebras.
\end{theorem}

Equivalently, we have

\begin{corollary}
\label{xxcor2} Suppose that an algebra $A$ has two graded algebra
decompositions
$$A=\bigoplus_{i=0}^{\infty} A_i=\bigoplus_{i=0}^{\infty} B_i$$
such that
\begin{enumerate}
\item[(1)]
$A_0=B_0=k$,
\item[(2)]
$A$ is generated by $A_1$ {\rm{(}}respectively, by $B_1${\rm{)}}, and 
\item[(3)]
either $A_1$ or $B_1$ is finite dimensional over $k$.
\end{enumerate}
Then there is an algebra automorphism $\phi: A\to A$ such that
$\phi(A_i)=B_i$ for all $i$.
\end{corollary}

Let $\Aut(A)$ (respectively, $\Aut_{gr}(A)$) be the group of algebra 
automorphisms (respectively, graded algebra automorphisms) of $A$.
We have an immediate consequence.

\begin{corollary}
\label{xxcor3}
Retain the hypotheses of Corollary {\rm{\ref{xxcor2}}}. If
$\Aut(A)=\Aut_{gr}(A)$, then $A_i=B_i$ for all $i$.
\end{corollary}

As an application, we give the following solution of an isomorphism problem.

Let $\{p_{ij} \mid i<j\}$ be a set of invertible scalars 
in $k^\times:=k\setminus \{0\}$. By convention, let $p_{ii}=1$ and $p_{ji}=
p^{-1}_{ij}$ if $i<j$.  Recall that the skew polynomial ring  $k_{p_{ij}}[x_i,\ldots,x_n]$ 
is generated by $x_1,\ldots,x_n$ and subject to the relations 
$$x_j x_i=p_{ij} x_i x_j$$
for all $i,j$.  An elementary change of generators means that we replace the 
ordered set $\{x_1,\ldots,x_n\}$ by $\{a_1 x_{\sigma(1)}, \ldots,
a_n x_{\sigma(n)}\}$ where $a_i\in k^\times$ and $\sigma$ is a permutation
in $S_n$. The following result is viewed as a partial generalization of 
\cite[Theorem 7.4]{Ga}.

\begin{theorem}
\label{xxthm0.4}
Suppose that $p_{ij}\neq 1$ for all $i\neq j$. Let $A$ be a graded ring
$k_{p_{ij}}[x_1,\ldots,x_n]/M$ where $M$ is an ideal 
in $k_{p_{ij}}[x_1,\ldots,x_n]_{\geq 3}$ and $B$ be a graded ring
$k_{p'_{ij}}[x_1,\ldots,x_m]/N$ where $N$ is an ideal 
in $k_{p'_{ij}}[x_1,\ldots,x_m]_{\geq 2}$. 
If $A$ is isomorphic to $B$ as ungraded algebras, then
$n=m$ and there is a permutation $\sigma\in S_n$ such that $p'_{ij}
=p_{\sigma(i)\sigma(j)}$ for all $i,j$. Furthermore, after an elementary
change of generators in $A$, $A=B$. 
\end{theorem}

\section{Proof of Theorem \ref{xxthm1}}
\label{xxsec1}

First we introduce an invariant of an algebra, which 
is similar to the Jacobson radical. Let $\dim$ denote the $k$-vector
space dimension. Let $s$ be a positive integer 
and define the Jacobson radical with tangent dimension $s$ to be
$$J_s(A)=\bigcap \{ I\mid I\subseteq A {\text{ such that }} 
\dim A/I=1 ~{\text{and}}\;  \dim I/I^2=s\}.$$
(In the case of an empty intersection, we take $J_s(A)=A$.)
An ideal $I$ of $A$ with the property
$\dim A/I=1, {\text{and}}\;  \dim I/I^2=s$ is called a codimension
1 ideal of tangent dimension $s$.
In general, given a sequence of non-negative integers
$(s_i)_{i\geq 0}$, we can define
$$J_{(s_i)}(A)=\bigcap \{ I\mid I\subseteq A, \dim I^i/I^{i+1}=s_i, {\text{for all
$i$}}\}.$$

The following lemma is easy.

\begin{lemma}
\label{xxlem1.1}
Let $A$ be an algebra finitely generated over $k$.
\begin{enumerate}
\item[(1)]
Suppose $A$ is generated by $d$ elements. 
If $s>d$, then $J_s(A)=A$.
\item[(2)]
Suppose $A$ is connected graded and $\dim A_{\geq 1}/(A_{\geq 1})^2=d$
{\rm{(}}this implies that $A$ is generated by $d$ elements{\rm{)}}.
Then $J_d(A)\subseteq A_{\geq 1}$.
\item[(3)]
Suppose that $k$ is infinite.
Then $J_{s+1}(A[t])= (J_s(A))[t]$.
\end{enumerate}
\end{lemma}

\begin{proof} (1) Let $\{f_1,\ldots,f_d\}$ be a set of
algebra generators of $A$. If $I$ is an ideal of codimension 1,
then $I$ is generated by $\{x_{i}\}_{i=1}^d$ where 
$x_i:=f_i-\alpha_i$ for some $\alpha_i\in k$. Then $A$
is generated by $\{x_i\}_{i=1}^d$ as an algebra and
$I^2$ is generated by $\{x_ix_j\}_{i,j=1}^d$. Hence
$\dim I/I^2 \leq d$. Thus there is no ideal $I$ of
codimension $1$ such that $\dim I/I^2=s>d$. The
assertion follows.

(2) Clear.

(3) Let $I$ be an ideal of $A$ of codimension 1 such that
$\dim I/I^2=s$. Let $I_{\alpha}$ be the ideal of $A[t]$
generated by $I$ and $t-\alpha$. Then $I_{\alpha}$
is of codimension 1 such that $\dim I_{\alpha}/I^2_{\alpha}
=s+1$. Using the fact that $k$ is infinite, it is easy to check that
\begin{equation}
\label{E1.1.1}\tag{E1.1.1}
\bigcap_{\alpha\in k} I_{\alpha}=I[t].
\end{equation}
Let $J$ be any idea of $A[t]$ of codimension 1 such that
$\dim J/J^2=s+1$. Let $I=J\cap A$. Then $I$ is an ideal 
of $A$ of codimension 1. There is always an $\alpha\in k$
such that $t-\alpha\in J$. Then $J$ is generated by $I$
and $t-\alpha$. So $J=I_{\alpha}$. Thus $\dim J/J^2=
\dim I/I^2 +1$. So $\dim I/I^2=s$. 
Now taking the intersection with all such $I$, equation 
\eqref{E1.1.1} implies that $J_{s+1}(A[t])= J_s(A)[t]$. 
\end{proof}

The next lemma is a special case of Theorem \ref{xxthm1}.
We prove it first as a warm-up.

\begin{lemma}
\label{xxlem1.2}
Let $A$ and $B$ be connected graded algebras generated in degree
1. Suppose that $d=\dim A_1<\infty$ and $A_{\geq 1}$ is the unique
codimension 1 ideal with tangent dimension $d$. If $\phi: A\to B$ 
is an isomorphism as ungraded algebras, then 
$A\cong B$ as graded algebras.
\end{lemma}

\begin{proof} By Lemma \ref{xxlem1.1}(1,2),
$\dim A_1=\dim B_1$. Since $A\cong B$, by hypothesis, there is
a unique ideal of $B$ of codimension 1 such that the tangent dimension
is $d$. Therefore $B_{\geq 1}$ is the
unique codimension 1 ideal of tangent dimension $d$.
Thus $\phi$ maps $I:=A_{\geq 1}$ to $L:=B_{\geq 1}$.
Therefore $\phi$ induces a graded algebra isomorphism
from 
$\gr_I A:=\bigoplus_{i=0}^{\infty} I^i/I^{j+1}$ to 
$\gr_L B:=\bigoplus_{i=0}^{\infty} L^i/L^{j+1}$.
Since $A\cong\gr_I A$ and $B\cong \gr_L B$ as graded algebras, there
is a graded algebra isomorphism from $A$ to $B$. 
\end{proof}

\begin{proof}[Proof of Theorem \ref{xxthm1}]
By Lemma \ref{xxlem1.1}(1,2), $\dim A_1=\dim B_1=:d<\infty$.
Pick $k$-linear bases of $A_1$ and $B_1$ respectively, 
say, $\{x_1,\ldots ,x_d\}$ and $\{y_1,\ldots,y_d\}$. 
Let $I=A_{\geq 1}$, $J=\phi(I)$, $L=B_{\ge 1}$, and $K=\phi^{-1}(L)$. 
Since $L=(y_1,\ldots,y_d)$ and $B$ is connected graded,
$L/L^2=\bigoplus_{i=1}^d k\overline{y_i}$.
Since $K$ is a codimension 1 ideal of tangent dimension
$d$, by making a linear change of the $x_i$, we may assume
that $K=(x_1,\ldots,x_{d-1}, x_d-\alpha)$, for some
$\alpha\in k$, and $K/K^2=
\bigoplus_{i=1}^{d-1} k\overline{x_i}\bigoplus 
k\overline{x_d-\alpha}$. The algebra isomorphism 
$\phi$ induces an $k$-linear isomorphism 
$$\overline{\phi}: K/K^2\to L/L^2.$$ By  making a 
linear change of the $y_i$, we have that
$$\overline{\phi} (\overline{x_i})=\overline{y_i}, 
\quad {\text{for all $i=1,\ldots,d-1$, \;  and}}
\quad
\overline{\phi} (\overline{x_d-\alpha})=\overline{y_d}.$$
Equivalently, we have that 
$\phi(x_i)=y_i+y_i'$ with $y_i'\in L^2$ 
for $i=1,\ldots ,d-1$ and that $\phi(x_d-\alpha) =
y_d + y_d'$ with $y_d'\in L^2$.

The easy case is when $\alpha=0$, then we have $J\subseteq L$.
Since $J$ has codimension 1, we then see that $J=L$. In fact, 
$\alpha=0$ if and only if $J=L$, if and only if $K=I$.
Since $\phi(I)=L$, the argument in the proof of 
Lemma \ref{xxlem1.2} shows that 
$$A\cong \gr_I(A)\cong \gr_L(B)\cong B.$$

The hard case is when $\alpha$ is nonzero (or equivalently, 
$J\neq L$ or $K\neq I$). 
Recall that $\phi(x_i)=y_i+y_i'$ with $y_i'\in L^2$ 
for $i=1,\ldots ,d-1$ and that 
$\phi(x_d)=\alpha+ y_d + y_d'$ with $y_d'\in L^2$. For 
$i=1,\ldots,d-1$ and $j\geq 1$, let $x_i^{[j]}$ denote 
the homogeneous element $[x_d,[x_d,\ldots,[x_d,x_i]\cdots]]$
where there are $j-1$ copies of $x_d$ appeared in the expression.
For example, $x_i^{[1]}=x_i$
for all $i=1,\ldots,d-1$. Similarly we define $y_{i}^{[j]}$.
Note that $x_i^{[j]}$ has degree $j$ and $\phi(x_i^{[j]})-y_i^{[j]}
\in L^{j+1}$. 

Let $r(x_i^{[j]})$
be a homogeneous relation of degree $m$ in $x_i^{[j]}$ for $1\leq i\leq d-1$,
$j\geq 1$.  
Then applying $\phi$, we see that
$r(y_i^{[j]})=r(y_i^{[j]})-\phi(r(x_i^{[j]}))\in L^{m+1}$ and hence
\begin{equation}
\label{E1.2.1}\tag{E1.2.1}
r(y_i^{[j]})=0
\end{equation}
in $B$.

In general, any homogeneous (noncommutative) polynomial 
$r:=r(x_1,\ldots ,x_d)$ in 
$x_1$, $\ldots$, $x_d$ of degree $m$ can be written as
\begin{equation}
\label{E1.2.2}\tag{E1.2.2}
r=\sum_{s=0}^m x_d^s r_s(x_{i}^{[j]})
\end{equation}
where $r_s(x_{i}^{[j]})$ is homogeneous polynomial of degree $m-s$
in $x_i^{[j]}$ for $1\leq i\leq d-1$ and $j\geq 1$.
Next we will prove the following two statements by induction:
\begin{enumerate}
\item[Claim 1:]
If $r$ is a homogeneous relation of degree $n$ in $x_1,\ldots,x_d$ 
(so $r(x_1,\ldots,x_d)=0$ in $A$), then $r(y_1,\ldots,y_d)=0$ in $B$.
\item[Claim 2:]
$\dim A_n=\dim B_n$.
\end{enumerate}
In fact, we will use induction on $n$ for all such isomorphisms $\phi: A\to B$
with the property that $\phi(I)\neq L$.
There is nothing to prove for $n=0$ and $1$. Now assume that Claim 1 and Claim 2 hold
for all $n<m$. To prove Claim 1 for $n=m$,
suppose that we have a homogeneous relation $r$ in 
$x_1,\ldots ,x_d$ of degree $m$ (so $r(x_1,\ldots,x_d)=0$ in $A$) and we 
have a decomposition $r=\sum_{s=0}^m x_d^s r_s$ as described 
in \eqref{E1.2.2}.  
We need show that $r(y_1,\ldots ,y_d)=0$. 

Let $t$ be the largest integer for which $r_t(x_{i}^{[j]})$ 
in \eqref{E1.2.2} is nonzero in $A$. (Equivalently the degree 
of $r_t(x_{i}^{[j]})$ is the smallest). If $t=0$, then it 
follows that $r(x_1,\ldots ,x_d)$ is a relation in $x_i^{[j]}$
But we have shown that all such relations have the property 
that $r(y_1,\ldots ,y_d)=0$, see \eqref{E1.2.1}, as desired. 
Now we assume that $t>0$. Applying $\phi$ to $r$, we see that 
\begin{equation}
\label{E1.2.3}\tag{E1.2.3}
0=\phi(r(x_1,\ldots ,x_d)) = 
\sum_{s=0}^t (y_d+\alpha+y_d')^s \phi(r_s(x_{i}^{[j]})).
\end{equation}
Then, by construction and \eqref{E1.2.3},
$\alpha^t r_t(y_{i}^{[j]})\in L^{m-t+1}$.  
This then gives by homogeneity that $r_t(y_{i}^{[j]})=0$.  
We have just shown that there is some homogeneous relation $r'$ in 
$y_1,\ldots ,y_d$ of degree $m-t$ such that 
$r'(x_1,\ldots ,x_d)\neq 0$.  Since $A_{m-t}$ and $B_{m-t}$ have the 
same dimension by Claim 2 (for $n=m-t$), it follows that there must exist
a relation $r''$ in $x_1,\ldots ,x_d$ of degree $m-t$ 
such that $r''(y_1,\ldots ,y_d)$ is nonzero.  This contradicts
Claim 1 for $n=m-t$. Therefore Claim 1 holds for $n=m$. 

As a consequence of Claim 1 for $n=m$, $\dim A_{m}\geq \dim B_{m}$. 
Since Claim 2 is independent of the choices of bases
$\{x_i\}_{i=1}^d$ and $\{y_i\}_{i=1}^d$, we obtain
$\dim B_m\geq \dim A_m$ by applying the consequence to 
the map $\phi^{-1}: B\to A$. Therefore Claim 2 holds for
$n=m$. This finishes the induction.

By Claims 1 and 2, since $A$ is graded, the map $\phi:A\to B$ 
defined by $\phi(x_i)=y_i$ gives a graded algebra isomorphism
from $A$ to $B$ as required.  
\end{proof}

Proofs of Corollaries \ref{xxcor2} and \ref{xxcor3} are easy 
and omitted.
Theorem \ref{xxthm1} fails when $A$ and $B$ are not generated
in degree 1.

\begin{example}
\label{xxex1.3} 
Let $A$ be the algebra $k_{-1}[x_1,x_2]\otimes k[y_1,y_2]$
with $\deg x_1=\deg x_2=1$ and $\deg y_1=\deg y_2=2$. 
Let $B$ be the algebra $k_{-1}[x_1,x_2]\otimes k[y_1,y_2]$
with $\deg x_1=\deg x_2=2$ and $\deg y_1=\deg y_2=1$. Then
$A\cong B$ as ungraded algebras, but not as graded algebras.
Further $A$ and $B$ have the same Hilbert series.
\end{example}

\section{Application 1: isomorphism problem}
\label{xxsec2}

In this section we consider the isomorphism problem for 
skew polynomial rings and their graded factor rings, and 
prove Theorem \ref{xxthm0.4}. 

\begin{lemma}
\label{xxlem2.1}
Let $A$ be the skew polynomial ring $k_{p_{ij}}[x_1,\ldots,x_n]$
and let $B=A/I$ where $I$ be a graded ideal of $A$ contained in
$A_{\geq 3}$. Suppose $p_{ij}\neq 1$ for all $i\neq j$. Then 
every normal element in $B_1$ is of the form $c x_i$ for some 
scalar $c$ and for some $1\leq i\leq n$.
\end{lemma}

\begin{proof} Let $f\in B_1=A_1$. If $I=0$, the assertion follows
from \cite[Lemma 3.5(d)]{KKZ}. Since $I\subseteq A_{\geq 3}$,
$f$ is normal in $A$ if and only if the image of $f$ is normal in $B=A/I$.
Therefore the assertion follows.
\end{proof}

Now we can prove Theorem \ref{xxthm0.4}.

\begin{proof}[Proof of Theorem \ref{xxthm0.4}] 
By Theorem \ref{xxthm1},
$A$ is isomorphic to $B$ as graded algebras. In particular, 
$n=m$. Let $\phi: B\to A$ be a graded algebra isomorphism. Since $\phi(x_i)$
are normal elements in $A$ of degree 1, by Lemma \ref{xxlem2.1},
$\phi(x_i)=c_i x_{\sigma(i)}$ for some $c_i\in k$ and some permutation 
$\sigma\in S_n$. Up to an elementary change of basis of $A$, $\phi$
sends $x_i$ to $x_i$ for all $i$. The assertions follow.
\end{proof}

\section{Application 2: cancellation problem}
\label{xxsec3}

We have another quick application. Let $Z(A)$ be the 
center of an algebra $A$.

\begin{theorem}
\label{xxthm3.1}
Let $A$ and $B$ be two connected graded algebras finitely generated in degree
1. Suppose that $Z(A)\cap A_1=\{0\}$. If $A[t_1,\ldots,t_n]
\cong B[s_1,\ldots ,s_n]$ as ungraded algebras, then $A\cong B$.
\end{theorem}

\begin{proof} If we set $\deg t_i=\deg s_i=1$ for all $i$, both
$C:=A[t_1,\ldots,t_n]$ and $D:=B[s_1,\ldots ,s_n]$ are connected graded 
and finitely generated in degree 1. Since $C\cong D$, by Theorem 
\ref{xxthm1}, there is a graded algebra isomorphism 
$\phi: C\cong D$. Since 
$Z(A)\cap A_1=\{0\}$, $Z(C)\cap C_1=\oplus_{i=1}^n k t_i$. By definition,
the $s_j$ are in the center of $D$.
Therefore $\phi^{-1}(s_j)\in \oplus_{i=1}^n k t_i$ for all $j$. By a dimension
argument, $\phi^{-1}(\oplus_{i=1}^n ks_i)= \oplus_{i=1}^n k t_i$.
Modulo $s_i$ and $t_i$ we obtain an induced algebra isomorphism
$\overline{\phi}: A\cong C/(t_i)\cong D/(s_i)\cong B$.
\end{proof}

Recall that an algebra $A$ is {\it cancellative}, if, for any algebra
$B$, an algebra isomorphism $A[t]\cong B[s]$ implies that $A
\cong B$. Theorem \ref{xxthm3.1} provides a weaker version 
of cancellation by assuming that both $A$ and $B$ are connected
graded finitely generated in degree 1. It would be interesting 
if one can improve Theorem \ref{xxthm3.1} by removing the 
hypothesis on $B$. 

\section*{Acknowledgments} 
J. Bell was supported by NSERC Discovery Grant 326532-2011.
J.J. Zhang was supported by the US National Science 
Foundation (grant Nos. DMS 0855743 and DMS 1402863).

\providecommand{\bysame}{\leavevmode\hbox to3em{\hrulefill}\thinspace}
\providecommand{\MR}{\relax\ifhmode\unskip\space\fi MR }
\providecommand{\MRhref}[2]{%
\href{http://www.ams.org/mathscinet-getitem?mr=#1}{#2} }
\providecommand{\href}[2]{#2}

\end{document}